\documentclass[11pt,letterpaper]{amsart}
\usepackage[utf8]{inputenc}
\usepackage{amsmath,amssymb,graphicx,setspace,verbatim,url}
\usepackage{color}
\usepackage{hyperref}
\usepackage{multicol}
\usepackage{enumitem}
\usepackage{float}
\usepackage{soul}
\usepackage{tikz}
\usepackage[normalem]{ulem}

\usepackage{xcolor}
\usetikzlibrary{matrix}
\usepackage[notcite,notref,color,]{showkeys}
\usepackage{xargs}

\usepackage{tikz-cd}
\usetikzlibrary{arrows.meta}
\newenvironment{red}{\relax\color{red}}{\hspace*{.5ex}\relax}
\newcommand{\bcp}{\begin{red} }
\newcommand{\ecp}{\end{red}}

\usepackage{comment}

\tikzcdset{
  arrow style={tikz,diagrams={>=Triangle}}
}

\input xy
\xyoption{all}

\newtheorem{prop}{Proposition}[section]

\newtheorem{thm}[prop]{Theorem}

\DeclareMathOperator{\Cor}{Cor}

\newcommand{\SSym}{Sym}
\newcommand{\Aut}{\mathrm{Aut}}

\usepackage[loadshadowlibrary,shadow, colorinlistoftodos,textsize=tiny,obeyFinal]{todonotes}

\setuptodonotes{fancyline, backgroundcolor=gray!10,bordercolor=gray}

\newcommandx{\inline}[2][1=]{\todo[inline, #1]{#2}}
\newcounter{hw}
\newcommandx{\homework}[2][1=]{\todo[inline, caption={Homework \thehw} #1]{\stepcounter{hw} #2}}

\newcommandx{\claudioLong}[2][1=Long todo, usedefault]{\todo[inline, color=red!25,caption={\textbf{Claudio:} #1}]{\textbf{Claudio: #1}. #2}}
\newcommandx{\philippeLong}[2][1=Long todo, usedefault]{\todo[inline,color=purple!25, caption={\textbf{Philippe:} #1}]{\textbf{Philippe: #1}. #2} }

\begin{document}
\title{Hypertopes with prescribed diagram symmetries}

\author[Dimitri Leemans]{Dimitri Leemans}
\address[Dimitri Leemans]{Université Libre de Bruxelles, 1050 Brussels, Belgium}
\email{leemans.dimitri@ulb.be}

\author[Pablo Spiga]{Pablo Spiga}
\address[Pablo Spiga]{Dipartimento di Matematica e Applicazioni,
Università di Milano-Bicocca, via Cozzi 55,
20126 Milan,
Italy}
\email{pablo.spiga@unimib.it}

\author[Philippe Tranchida]{Philippe Tranchida}
\address[Philippe Tranchida]{Université Libre de Bruxelles, 1050 Brussels, Belgium}
\email{tranchida.philippe@gmail.com}


\setcounter{tocdepth}{1}

\date{\today}

\begin{abstract}
Let $\mathcal{G}$ be a finite connected simple graph with at least two vertices. We construct a finite regular hypertope $\Gamma$ whose diagram is the graph $\mathcal{G}$ with every edge labelled by $4$, and for which
$\Cor(\Gamma)/\Aut(\Gamma) \cong \Aut(\mathcal{G})$. This shows that any possible group of diagram symmetries can be realized by the correlations of a hypertope.
The construction uses a group generated by involutions, with nilpotency class two, exponent four and commutator relations dictated by the defining graph.
\end{abstract}

\maketitle
\section{Introduction}
An incidence geometry is a set of elements, each of which is assigned a type, together with an incidence relation. They permit to encode a wide variety of geometrical structures, such as concrete or abstract polytopes, affine and projective spaces, Tits' buildings and maps or hypermaps on surfaces, using a unified language. 

An automorphism of an incidence geometry is an incidence preserving symmetry of its elements that also preserves the type of every element, while a correlation is allowed to permute the types. If $I$ is the type set of a geometry $\Gamma$, the action on $I$ gives a homomorphism
\[
\theta\colon \Cor(\Gamma)\longrightarrow \SSym(I)
\]
with kernel $\Aut(\Gamma)$. Thus $\Aut(\Gamma)$ is a normal subgroup of $\Cor(\Gamma)$ and  $\Cor(\Gamma)/\Aut(\Gamma)$ is the group of permutations of types that are induced by correlations.
This distinction is also used by Leemans, Stokes and Tranchida in \cite{leemans2025groups}, where incidence geometries are used to represent inner and outer automorphisms of groups. Note that, in \cite{leemans2025groups}, the correlation and automorphism groups of the geometry are prescribed. Here we are only interested in the quotient group.

We focus on a specific instance of incidence geometries, called regular hypertopes. Hypertopes are thin, residually connected incidence geometries, and a hypertope $\Gamma$ is regular when its type-preserving automorphism group $\Aut(\Gamma)$ acts flag-transitively, achieving the highest possible degree of symmetry for a hypertope. Hypertopes were introduced by Fernandes, Leemans and Weiss in \cite{hypertopes} as a generalization of abstract polytopes. Regular polytopes and hypertopes can be constructed directly from groups. Polytopes are built from string $C$-groups, whereas hypertopes allow diagrams that are not strings and are built from $C$-groups. Since the diagram of a polytope is always linear, the only possible non trivial symmetry of its type set induced by correlations is a duality, exchanging the role of vertices and facets. On the other hand, any labelled graph can be the diagram of a hypertope, so that the symmetries of its type set can in theory be any group. The goal of this article is to show that any finite group $H$ can indeed be achieved as the correlation-induced group of symmetries of the types of a regular hypertope.

When $H \cong C_2$, the cyclic group of order two, the question is answered positively by any self-dual polytope, for instance. When $H \cong C_3$, the question becomes the construction of geometries with trialities but no dualities. Leemans, Stokes and Tranchida \cite{TrialitySuzuki} constructed and infinite family of finite regular hypertopes of rank three for which the induced group on types is cyclic of order three. Further families of such geometries were constructed in \cite{delaby2026geometriesadmittingtrialitiessymmetric}. Thick incidence geometries with trialities but no dualities where constructed in~\cite{ALCO_2026__9_2_379_0}. Examples for $H \cong \SSym(3)$ are much easier to construct, for instance  by considering triples of tangent line to a conic in a finite projective plane, see \cite{tranchida2025triples}.

Our construction for a generic $H$ starts with a finite connected simple graph $\mathcal{G}$ with at least two vertices and the idea of the construction was inspired by the content of \cite{Group-vertex-edge}. We take one involution for each vertex, impose that the nilpotency class is equal to two, so that all commutators are central, and make generators commute when the corresponding vertices are not adjacent. The resulting group has a normal form with one binary coordinate for each vertex and one for each edge of the graph. We use these coordinates to verify the intersection property and flag-transitivity of its standard coset incidence system.
Theorems~\ref{thm:Main} and~\ref{thm:correlations} show that this gives a finite regular hypertope $\Gamma$ of rank $|V(\mathcal{G})|$, with diagram $\mathcal{G}$ having all edges labelled by $4$, and with
\[
\Cor(\Gamma)/\Aut(\Gamma)\cong\Aut(\mathcal{G}).
\]
Every graph automorphism extends by permuting the generators. Conversely, the commutation relations between these generators recover the graph, so no other type permutations occur. Frucht's theorem \cite{Frucht1939} then gives the required realization of every finite group $H$. 

    \begin{thm}\label{maintheo}
        For any finite group $H$, there exists a regular hypertope $\Gamma$ such that $\Cor(\Gamma)/\Aut(\Gamma) \cong H$.
    \end{thm}

\section{Background}
All sets and graphs in this paper are finite. A simple graph means an undirected graph with no loops nor multiple edges. We set $[x,y]=x^{-1}y^{-1}xy$, and write $G'=[G,G]$ for the commutator subgroup of a group $G$. We now give a quick background on the necessary terminology for incidence systems, geometries, and regular hypertopes.

\subsection{Incidence systems}
We recall the terminology needed below; a general reference is \cite[Chapter~1]{buekenhout2013diagram}.
An \emph{incidence system} is a quadruple $\Gamma=(X,I,*,t)$, where $X$ is a set of elements, $I$ is a set of types, $t\colon X\to I$ is surjective, and $*$ is a reflexive symmetric relation on $X$. Moreover, We require that two distinct incident elements have different types. The \emph{rank} of $\Gamma$ is $|I|$.

A \emph{flag} is a set of pairwise incident elements. Its type is the set of types of its elements, and a \emph{chamber} is a flag of type $I$. An incidence system is an \emph{incidence geometry} if every flag is contained in a chamber. The \emph{residue} of a flag $F$ consists of the elements of types outside $t(F)$ that are incident with every element of $F$, with the induced incidence relation and type map. Its rank is $|I\setminus t(F)|$.

The \emph{incidence graph} has vertex set $X$, with an edge between distinct incident elements. A geometry is \emph{residually connected} if every residue of rank at least two has a connected incidence graph. The empty flag is included, so this condition also applies to the whole geometry when its rank is at least two. A geometry is \emph{thin} if each residue of rank one contains exactly two elements.

A \emph{correlation} of $\Gamma$ is a bijection $\alpha\colon X\to X$ preserving incidence in both directions, for which there is a permutation $\sigma$ of $I$ satisfying $t(\alpha(x))=\sigma(t(x))$ for every $x\in X$. It is an \emph{automorphism} if $\sigma$ is the identity. The group of all correlations is denoted by $\Cor(\Gamma)$ and the group of all automorphisms by $\Aut(\Gamma)$. To each correlation $\alpha$ we can assign its permutation $\sigma$ of the type set $I$. Doing so yields a homomorphism $\theta \colon \Cor(\Gamma) \to \SSym(I): \alpha \mapsto \sigma$ with kernel $\Aut(\Gamma)$. Consequently,
\[
\Cor(\Gamma)/\Aut(\Gamma)\cong\theta(\Cor(\Gamma))
\leq\SSym(I).
\]
A group of automorphisms is \emph{flag-transitive} if it is transitive on the flags of each fixed type. For a geometry, this is equivalent to transitivity on chambers.

\subsection{Regular hypertopes}
A \emph{hypertope} is a thin, residually connected incidence geometry. It is \emph{regular} if its automorphism group is flag-transitive \cite{hypertopes}. In a thin geometry, each chamber $C$ has a unique distinct chamber $C^i$ differing from $C$ only in its element of type $i$. These chambers are called $i$-adjacent. The \emph{chamber graph} has the chambers as vertices and an edge labelled $i$ between $i$-adjacent chambers.

The chamber graph of a residually connected geometry is connected; see \cite[Chapters~1 and~3]{buekenhout2013diagram}. In a thin geometry with connected chamber graph, an automorphism fixing a chamber fixes each adjacent chamber, and hence all chambers. It is therefore the identity. In particular, the automorphism group of a regular hypertope acts simply transitively on chambers.

A \emph{$C$-group} is a pair $(G,S)$, where $S=\{\rho_i\mid i\in I\}$ is a generating set of involutions satisfying the \emph{intersection property}
\[
G^J\cap G^K=G^{J\cap K}\qquad(J,K\subseteq I),
\qquad
G^J:=\langle\rho_j\mid j\in J\rangle.
\]
Here $G^\emptyset=\{1\}$. If $\Gamma$ is a regular hypertope and $C$ is a chamber, the unique automorphisms $\rho_i$ sending $C$ to $C^i$ generate $G=\Aut(\Gamma)$ and form a $C$-group \cite[Theorem~4.1]{hypertopes}. We call them the \emph{distinguished generators} with respect to $C$.
Conversely, given a group $G$ generated by involutions $S=\{\rho_i\mid i\in I\}$, set
\[
G_i=\langle\rho_j\mid j\in I\setminus\{i\}\rangle.
\]
The \emph{standard coset incidence system} $\Gamma(G,(G_i)_{i\in I})$ has as its elements of type $i$ the right cosets $G_i g$, and two cosets are incident when their intersection is nonempty. The group $G$ acts by right multiplication, and $\{G_i\mid i\in I\}$ is the standard chamber. This construction is due to Tits~\cite{Tits1957} (see also \cite[Section~3]{hypertopes}).

For $J\subseteq I$, write
\[
G_J:=\bigcap_{j\in J}G_j,\qquad G_\emptyset:=G.
\]
The subgroups $G_i$ are called the \emph{maximal parabolic subgroups}. When $(G,S)$ is a $C$-group, its intersection property gives $G_J=G^{I\setminus J}$. The notations $G_J$ and $G^J$ thus use complementary sets of indices.

The following criterion is \cite[Theorem~1.8.10]{buekenhout2013diagram}, with the subgroup products reversed by taking inverses.
\begin{prop}\label{prop:FT}
The group $G$ is flag-transitive on $\Gamma(G,(G_i)_{i\in I})$ if and only if
\[
G_iG_J=\bigcap_{j\in J}G_iG_j
\qquad\text{for all }J\subseteq I\text{ and }i\in I\setminus J.
\]
If this condition holds, the coset incidence system is a geometry.
\end{prop}

A $C$-group need not satisfy this additional condition but, when it does, it corresponds to a regular hypertope \cite[Theorem~4.6]{hypertopes}.
\begin{prop}\label{prop:hypertope}
If $(G,S)$ is a $C$-group and $G$ is flag-transitive on its standard coset incidence system $\Gamma$, then $\Gamma$ is a regular hypertope. Its automorphism group is isomorphic to $G$.
\end{prop}

The \emph{Buekenhout diagram}, or \emph{Coxeter diagram}, of a regular hypertope has vertex set $I$. Vertices $i$ and $j$ are joined by an edge labelled
\[
m_{i,j}=o(\rho_i\rho_j)
\]
when $m_{i,j}>2$; an edge is omitted when $m_{i,j}=2$. This agrees with the diagram defined by the rank-two residues, where each residue of type $\{i,j\}$ is an $m_{i,j}$-gon \cite[Section~4]{hypertopes}. 

Finally, correlations can be read from automorphisms of the group that permute its distinguished generators \cite[Proposition~3.15]{TrialitySuzuki}. Indeed, after composing a correlation with a type-preserving automorphism, we may assume that it fixes the standard chamber. Conjugation then sends $\rho_i$ to $\rho_{\sigma(i)}$, where $\sigma$ is its type permutation. Conversely, a group automorphism $\beta$ with $\beta(\rho_i)=\rho_{\sigma(i)}$ sends $G_i$ to $G_{\sigma(i)}$ and defines a correlation by
\[
G_i g\longmapsto G_{\sigma(i)}\beta(g).
\]
Thus, setting
$\Aut(G,S):=\{\beta\in\Aut(G)\mid\beta(S)=S\}$,
we obtain
\[
\Cor(\Gamma)/\Aut(\Gamma)\cong\Aut(G,S).
\]

\section{Regular Hypertopes with Prescribed Diagram Symmetries}

Let $\mathcal{G}$ be any given finite connected simple graph with vertex set $V =V(\mathcal{G}) $ and edge set $E = E(\mathcal{G})$. Assume that $\mathcal{G}$ has at least two vertices. We construct a group $G$ associated to $\mathcal{G}$ in the following way.

Label the vertices of $\mathcal{G}$ by integers going from $1$ to $n = |V|$. The group $G$ is set to be generated by $n$ involutions $\rho_1,\rho_2, \ldots,\rho_n$, where the involution $\rho_i$ corresponds to the $i$'th vertex of $\mathcal{G}$. Moreover, $G$ is required to be a nilpotent group of nilpotency class two\footnote{A group of nilpotency class two is a group whose commutator subgroup is contained in its centre.}. Finally, whenever there is no edge between two vertices $i$ and $j$, we impose that $\rho_i$ and $\rho_j$ commute.

Since $G$ is of nilpotency class two, it means that the commutator subgroup $[G,G]$ is central in $G$. In other words, commutators commute with every elements of $G$. Moreover, since $G$ is generated by involutions, all commutators also have order two. Indeed, we have $[\rho_i,\rho_j]^2 = [\rho_i,\rho_j]\rho_i\rho_j\rho_i\rho_j = \rho_i[\rho_i,\rho_j]\rho_j\rho_i\rho_j = e$ and the group $[G,G]$ is abelian, generated by the set of commutators $\{[\rho_i,\rho_j] \mid i < j \in I\}$, where we set once for all $I = \{1,2,\ldots,n\}$.

In other words, the group $G$ is defined by the following presentation, where we set $I = \{1,2,\ldots,n\}$.
\[
G :=
\left\langle
\rho_i \ (i\in I)
\ \middle|\
\begin{array}{ll}
\rho_i^2 = 1
& \text{for all } i\in I,\\[2mm]
[[\rho_i,\rho_j],\rho_k]=1
& \text{for all } i,j,k\in I,\\[2mm]
[\rho_i,\rho_j]=1
& \text{whenever } \{i,j\}\notin E
\end{array}
\right\rangle .
\]

Set $c_{i,j} =[\rho_i, \rho_j]$. Then, the previous discussion shows that any element $g \in G$ can be written as 

\begin{equation} \label{eq:normalForm}
g =\rho_1^{\epsilon_1}\rho_2^{\epsilon_2} \ldots \rho_n^{\epsilon_n} \prod_{\substack{\{i,j\}\in E\\ i<j}} c_{i,j}^{\epsilon_{i,j}}
\end{equation}

with $\epsilon_i, \epsilon_{i,j} \in \{0,1\}$. This shows that $G$ also has exponent four. Indeed, from Equation~(\ref{eq:normalForm}) and the above discussion, we get that $g^2 \in [G,G]$ for any $g \in G$, so that $g^4 = e$. Moreover, it is clear from the normal form (\ref{eq:normalForm}) that $|G| = 2 ^{(|V|+|E|)}$.

Equation~(\ref{eq:normalForm}) actually gives a strong normal form for elements of $G$, meaning that any $g \in G$ is uniquely expressed in such a way. The exponents $\epsilon_i$ are seen to be uniquely determined by using the abelianization homomorphism $\pi \colon G \to G^{ab} = G/[G,G] \cong C_2^n$. Similarly, the exponents $\epsilon_{i,j}$ are verified to be unique by looking at the images of an element $g$ under the homomorphisms $\pi_{i,j} \colon G \to D_4$, where $D_4 = \langle a,b \rangle$ is the dihedral group with 8 elements generated by two involutions $a$ and $b$, and $\pi_{i,j}$ sends $\rho_i$ to $a$, $\rho_j$ to $b$ and the other generators of $G$ to the identity.

\begin{thm}\label{thm:Main}
    Let $\mathcal{G}$ be a simple graph, $G$ be the class two nilpotent group associated to it and $\Gamma = (G,(G_i)_{i \in I})$ be the standard coset geometry associated to $G = \langle \rho_1, \rho_2, \ldots, \rho_n \rangle$, where $G_i$ is the subgroup of $G$ generated by all $\rho_j$'s except from $\rho_i$. Then, $\Gamma$ is a regular hypertope, whose diagram is the graph $\mathcal{G}$, where all edges are labelled by $4$.
\end{thm}

\begin{proof}
    We first show that $G$ satisfies the intersection property, and therefore that $(G,S)$, with $S = \{\rho_1,\rho_2, \ldots, \rho_n\}$ is a $C$-group. For any two subsets $J,K \subset I = \{1,2,\ldots, n\}$, we need to show that 

    \begin{equation}\label{eq:intersectionProp}
        G^J \cap G^K = G^{J \cap K}
    \end{equation} 

    where $G^L := \langle \rho_i \mid i \in L\rangle$ for any subset $L \subset I$. By the normal form given in Equation~(\ref{eq:normalForm}), it is immediate that elements of $G^J$ consists of those whose normal form satisfies $\epsilon_i = 0$ for $i \notin J$ and $\epsilon_{i,j} = 0$ unless $\{i,j\} \subset J$. Since this holds for $G^K$ and $G^{J \cap K}$ also, Equation~(\ref{eq:intersectionProp}) trivially holds.

    Now, we will show that $\Gamma$ is flag-transitive. We need to show that

    \begin{equation}\label{eq:FT}
        G_iG_J = \bigcap_{j \in J} G_iG_j
    \end{equation}
    for any $J \subset I$ and where $G_J := \bigcap_{j \in J} G_j$. Note that we can assume that $i \notin J$ since if $i \in J$, both side of Equation~(\ref{eq:normalForm}) are equal to $G_i$. We thus suppose that $i \notin J$ for the rest of the proof.
    We need to understand normal forms of elements in products $G_iG_j$. Elements of $G_i$ are precisely those whose normal form has $\epsilon_i = 0$ and $\epsilon_{ik} = 0$ for any $k \in I$. Since the commutators are central and since $\rho_i \in G_j$, we can drop the condition $\epsilon_i = 0$ for elements of $G_iG_j$. Moreover, as $c_{ik} \in G_j$ for $k \neq j$, we can also drop the conditions $\epsilon_{ik} = 0$, except for $e_{i,j} = 0$ when $\{i,j\} \in E$. In particular, if $\{i,j\} \notin E$, we have $G_iG_j = G$. Let us now suppose, without loss of generality, that $i$ is the last element of the ordering in the normal form. Indeed, in Equation~(\ref{eq:normalForm}), we ordered indices naturally by $1 < 2< \ldots <n$, but we could choose any other ordering and obtain another valid normal form. Using a normal form where $i$ is the last index, we can then see that if $\{i,j\} \in E$, the set $G_iG_j$ consists of all elements whose normal form satisfies $\epsilon_{i,j} = 0$. Note that the choice on the ordering of elements of $I$ matters for this description, because if we had needed to switch $\rho_i$ and $\rho_j$ an odd number of times to re-order the generators, we would have introduced a commutator $[\rho_i,\rho_j] = c_{i,j}$. Note that the exact same reasoning also shows that normal forms for elements in $G_iG_J$ are precisely those with $\epsilon_{i,j} = 0$ for any $\{i,j\}$ with $j \in J$ and $\{i,j\} \in E$.

    The inclusion $G_iG_J \subset \cap_{j \in J} G_iG_j$ always holds. It suffices thus to show the reverse inclusion. Let $g \in \cap_{j \in J}G_iG_j$. Thus $g \in G_iG_j$ for every $j \in J$. Therefore, according to the above discussion, the normal form for $g$ is characterized by $\epsilon_{i,j} = 0$ for every $j \in J$ such that $\{i,j\} \in E$. As explained above, any elements of this form is contained in $G_iG_J$, as desired.
\end{proof}

\begin{thm}\label{thm:correlations}
    Let $\mathcal{G}$ be a simple graph and let $\Gamma$ be the coset geometry associated to it constructed in Theorem \ref{thm:Main}. Then, $\Cor(\Gamma)/\Aut(\Gamma) \cong \Aut(\mathcal{G})$.
\end{thm}

\begin{proof}
    Let $\gamma \in \Aut(\mathcal{G})$ be an automorphism of $\mathcal{G}$. This induces a permutation of the generators of $G$ by setting $\alpha_\gamma (\rho_i) =  \rho_{\alpha(i)}$ for every $i \in I$. This map $\alpha_\gamma$ extends to an homomorphism of $G$ because it preserves all relations imposed on $G$. Since $\alpha_\gamma$ only permutes the generators, it is then an isomorphism of $G$. This isomorphism sends the maximal parabolic subgroup $G_i$ to $G_{\alpha_i}$ and thus induces a well defined correlation of $\Gamma$.

    So far, we have thus shown that there is an injective homomorphism $\alpha \colon \Aut(\mathcal{G}) \to \Cor(\Gamma)$. The correlation $\alpha_\gamma$ preserves types if and only $\gamma = e$, so that the image of $\alpha$ is actually contained in $\Cor(\Gamma)/\Aut(\Gamma) < \SSym(I)$. By \cite[Proposition 3.15]{TrialitySuzuki}, we know that any symmetry in $\Cor(\Gamma)/\Aut(\Gamma)$ must precisely come from an automorphism of $G$ permuting the generators, showing that $\alpha$ is surjective.\end{proof}

    \begin{proof}[Proof of Theorem~\ref{maintheo}]
    Frucht's theorem (\cite{Frucht1939}) states that the finite group $H$ can be realized as $\Aut(\mathcal{G})$ for some finite connected simple graph $\mathcal{G}$. Theorems \ref{thm:Main} and \ref{thm:correlations} then prove that there exists a regular hypertopes with $\Cor(\Gamma)/\Aut(\Gamma) \cong \Aut(\mathcal{G}) \cong H$.     
    \end{proof}

The hypertope constructed is this way for $n=3$ and graph $\mathcal{G}$ a complete graph with labels 4 on all edges has an automorphism group of order $2^{3+3}=64$. This is much smaller than the order of the groups constructed in~\cite{HFL2019}, namely 1024.
A search in the lists of hypermaps available on Marston Conder's website\footnote{\url{https://www.math.auckland.ac.nz/~conder/}} shows that the hypermaps with same diagram and smaller groups are not hypertopes. This is easy to see as a hypertope with such a diagram must have at least 40 chambers and the groups of hypermaps RPH2.4, RPH3.6 and RPH3.7 are smaller than that. Hence the hypertope we obtain in this case is the smallest possible one for the given diagram. We now prove that if $G$ is a 2-group that is the automorphism group of a regular hypertope of rank $n$ with diagram a complete graph with label 4 on each edge, then the order of $G$ is at least $2^{n+{n\choose 2}}$, making the hypertopes constructed here the smallest ones for 2-groups.

\begin{thm}
    Let $\mathcal{G}$ be a complete graph with $n$ vertices and with label four on every edge and assume that $G$ is a $2$-group that is the automorphism group of a regular hypertope of rank $n$ whose diagram is $\mathcal{G}$. Then $G$ must have at least $2^n$ elements of each type, and $|G|$ is at least $2^{n+{n\choose 2}}$.
\end{thm}
\begin{proof}
    We prove the fact that $G$ has at least $2^n$ elements of each type by induction on the rank $n$.
    The cases where $n=1$ and $n=2$ are obvious. The case where $n=3$ was discussed already in the paragraph above.
    
    Suppose the theorem is true for $n$ and let us show that it must then be true for $n+1$.
    Take a hypertope $\Gamma$ of rank $n+1$ satisfying the hypotheses of the theorem and let $G$ be the automorphism group of $\Gamma$. Take an element $x$ of the hypertope $\Gamma$ and let $t(x)$ be its type. For any element $y$ of $\Gamma$ of type $t(y) \neq t(x)$, we can consider its residue $\Gamma_y$, which is a hypertope of rank $n$ satisfying the conditions of the theorem. By induction hypothesis, there are at least $2^n$ elements of type $t(x)$ in $\Gamma_y$. Note that the number of elements of type $t(x)$ in $\Gamma$ is given by $[G : G_x] =|G|/|G_x|$, which must be a power of two since $G$ is a $2$-group. Hence, if there are strictly more that $2^n$ elements of type $t(x)$ in $\Gamma$, there must be at least $2^{n+1}$ elements of type $t(x)$ in $\Gamma$. If there were exactly $2^n$ elements of type $t(x)$ in $\Gamma$, it would mean, by the preceding count, that they are all incident to $y$. Since this argument holds for any $y$ of type $t(y)$, it implies that all elements of type $t(x)$ are incident to all elements of type $t(y)$. Hence, there should be a two in the diagram on the edge joining $t(x)$ and $t(y)$, a contradiction.

    The fact that $|G| \geq 2^{n + {n \choose 2}}$ now follows from a simple argument. Let $c = \{x_1,x_2,\ldots,x_{n}\}$ be a chamber of $\Gamma$ and let $c_i = \{x_i, \ldots,x_n\}$ for $i=1, \ldots, n$ and set $c_{n+1} = \emptyset$. Then, by the first part of the theorem, $[G_{c_2}:G_{c_1}] \geq 2$, $[G_{c_3}:G_{c_2}] \geq 4$, and, more generally, $[G_{c_{i+1}}:G_{c_i}] \geq 2^{i}$ for any $i = 1,2,\ldots,n+1$, and where $G_\emptyset = G$. Therefore, we conclude that $|G| \geq 2^{1+2+\cdots +n} = 2^{n(n+1)/2} = 2^{n+{n\choose 2}}$.
\end{proof}
Observe that the previous theorem ensures that no proper quotient of the hypertopes we constructed with $\mathcal{G}$ a complete graph with label 4 on every edge can themselves be hypertopes.
\section{Acknowledgements}
The authors thank BIRS-CMO for the organisation of the workshop 26w5615-"Restricted Symmetry and Asymptotic Enumeration of Discrete Structures" during which this research was done.
Dimitri Leemans acknowledges financial support from the Actions de Recherche Concertées of the Fédération Wallonie-Bruxelles. Philippe Tranchida is a Postdoctoral Researcher of the Fonds de la Recherche Scientifique -- FNRS.

\bibliographystyle{siam}
\bibliography{refs}

@book{buekenhout2013diagram,
  title={Diagram geometry: related to classical groups and buildings},
  author={Buekenhout, Francis and Cohen, Arjeh M},
  volume={57},
  year={2013},
  publisher={Springer Science \& Business Media}
}

@incollection {Tits1957,
    AUTHOR = {Tits, J.},
     TITLE = {Sur les analogues alg\'{e}briques des groupes semi-simples
              complexes},
 BOOKTITLE = {Colloque d'alg\`ebre sup\'{e}rieure, tenu \`a {B}ruxelles du 19 au 22
              d\'{e}cembre 1956},
    SERIES = {Centre Belge de Recherches Math\'{e}matiques},
     PAGES = {261--289},
 PUBLISHER = {\'{E}tablissements Ceuterick, Louvain},
      YEAR = {1957},
   MRCLASS = {50.00 (22.00)},
  MRNUMBER = {108765},
MRREVIEWER = {H. Freudenthal},
}

@article {hypertopes,
    AUTHOR = {Fernandes, Maria Elisa and Leemans, Dimitri and Weiss, Asia
              Ivi\'{c}},
     TITLE = {Highly symmetric hypertopes},
   JOURNAL = {Aequationes Math.},
  FJOURNAL = {Aequationes Mathematicae},
    VOLUME = {90},
      YEAR = {2016},
    NUMBER = {5},
     PAGES = {1045--1067},
      ISSN = {0001-9054},
   MRCLASS = {51E24 (20F05 52B11)},
  MRNUMBER = {3547707},
MRREVIEWER = {Hendrik Van Maldeghem},
       DOI = {10.1007/s00010-016-0431-1},
       URL = {https://doi.org/10.1007/s00010-016-0431-1},
}

@article {TrialitySuzuki,
    AUTHOR = {Leemans, Dimitri and Stokes, Klara and Tranchida, Philippe},
     TITLE = {Flag transitive geometries with trialities and no dualities
              coming from {S}uzuki groups},
   JOURNAL = {J. Combin. Theory Ser. A},
  FJOURNAL = {Journal of Combinatorial Theory. Series A},
    VOLUME = {213},
      YEAR = {2025},
     PAGES = {Paper No. 106033},
      ISSN = {0097-3165,1096-0899},
   MRCLASS = {99-06},
  MRNUMBER = {4870227},
       DOI = {10.1016/j.jcta.2025.106033},
       URL = {https://doi.org/10.1016/j.jcta.2025.106033},
}

@article{Frucht1939,
  author  = {Frucht, Robert},
  title   = {Herstellung von {Graphen} mit vorgegebener abstrakter {Gruppe}},
  journal = {Compositio Mathematica},
  volume  = {6},
  year    = {1939},
  pages   = {239--250},
  url     = {https://www.numdam.org/item/CM_1939__6__239_0/}
}

@article{leemans2025groups,
  title={Groups represented by incidence geometries},
  author={Leemans, Dimitri and Stokes, Klara and Tranchida, Philippe},
  journal={arXiv preprint arXiv:2506.10464},
  year={2025}
}

@article{ALCO_2026__9_2_379_0,
     author = {Delaby, Remi and Leemans, Dimitri and Tranchida, Philippe},
     title = {Geometries with trialities arising from linear spaces},
     journal = {Algebraic Combinatorics},
     pages = {379--402},
     year = {2026},
     publisher = {The Combinatorics Consortium},
     volume = {9},
     number = {2},
     doi = {10.5802/alco.482},
     language = {en},
     url = {https://alco.centre-mersenne.org/articles/10.5802/alco.482/}
}

@misc{delaby2026geometriesadmittingtrialitiessymmetric,
      title={Geometries admitting trialities for the symmetric and alternating groups}, 
      author={Rémi Delaby},
      year={2026},
      eprint={2607.19230},
      howpublished={Preprint},
      archivePrefix={arXiv},
      primaryClass={math.CO},
      url={https://arxiv.org/abs/2607.19230}, 
}

@article{tranchida2025triples,
  title={Triples of involutions in ${PGL}(2, q)$ and their incidence geometries},
  author={Tranchida, Philippe},
  journal={Innovations in Incidence Geometry: Algebraic, Topological and Combinatorial},
  volume={22},
  number={1},
  pages={25--46},
  year={2025},
  publisher={Mathematical Sciences Publishers}
}

@article {HFL2019,
    AUTHOR = {Hou, Dong-Dong and Feng, Yan-Quan and Leemans, Dimitri},
     TITLE = {Existence of regular 3-hypertopes with {$2^n$} chambers},
   JOURNAL = {Discrete Math.},
  FJOURNAL = {Discrete Mathematics},
    VOLUME = {342},
      YEAR = {2019},
    NUMBER = {6},
     PAGES = {1857--1863},
      ISSN = {0012-365X,1872-681X},
   MRCLASS = {52B05},
  MRNUMBER = {3924410},
MRREVIEWER = {Egon\ Schulte},
       DOI = {10.1016/j.disc.2019.02.022},
       URL = {https://doi-org.ezproxy.ulb.ac.be/10.1016/j.disc.2019.02.022},
}

@article {Group-vertex-edge,
    AUTHOR = {Li, Yinan and Qiao, Youming},
     TITLE = {Group-theoretic generalisations of vertex and edge
              connectivities},
   JOURNAL = {Proc. Amer. Math. Soc.},
  FJOURNAL = {Proceedings of the American Mathematical Society},
    VOLUME = {148},
      YEAR = {2020},
    NUMBER = {11},
     PAGES = {4679--4693},
      ISSN = {0002-9939,1088-6826},
   MRCLASS = {20D15 (05C40 15A69)},
  MRNUMBER = {4143386},
MRREVIEWER = {Pablo\ Spiga},
       DOI = {10.1090/proc/15184},
       URL = {https://doi.org/10.1090/proc/15184},
}

\end{document}